\documentclass[12pt]{article}
\usepackage{amssymb,amsmath,latexsym,amsthm}
\usepackage{epigraph}
\usepackage{picinpar}
\usepackage{graphicx}
\usepackage{dsfont}

\theoremstyle{plain}
\newtheorem{theorem}{Theorem}
\newtheorem{corollary}[theorem]{Corollary}
\theoremstyle{definition}

\newcommand{\R}[1]{$\mathbb{R}^#1$}
\newcommand{\V}[1]{\vec{\textbf{#1}}}

\begin{document}
\title{Redimensioning of Euclidean Spaces}
\author{Piyush Ahuja\thanks{Piyush Ahuja will graduate from IIT Delhi in 2013 with a major in Maths and Computing. His current research straddles the fields of Algorithms and Economics. The work presented here was concluded in his sophomore year.} \and Subiman Kundu\thanks{Dr. Subiman Kundu is \ldots  }}
%\affil{Department of Mathematics,\\ Indian Institute of Technology Delhi}
\date{}
\maketitle
\abstract{A vector space over a field $\mathbb{F}$ is a set $V$ together with two binary operations, called vector addition and scalar multiplication. It is standard practice to think of a Euclidean space $\mathbb{R}^n$ as an $n$-dimensional real coordinate space i.e. the space of all $n$-tuples of real numbers ($R^n$), with vector operations defined using real addition and multiplication coordinate-wise. A natural question which arises is if it is possible to redefine vector operations on the space in such a way that it acquires some other dimension, say $k$ (over the same field i.e., $\mathbb{R}$). In this paper, we answer the question in the affirmative, for all $k\in\mathbb{N}$. We achieve the required dimension by `dragging' the structure of a standard $k$-dimensional Euclidean space (\R{k}) on the $n$-tuple of real numbers ($R^n$). At the heart of the argument is Cantor's counterintuitive result that $\mathbb{R}$ is numerically equivalent to $\mathbb{R}^n$ for all $n\in\mathbb{N}$, which can be proved through an elegant construction. Finally, we generalize the result to all finite dimensional vector spaces.

}

\section{Introduction}

\begin{quote}
\textit{An engineer, a physicist, and a mathematician are discussing how to visualise four dimensions: \\
Engineer: I never really get it.\\
Physicist: Oh it's really easy, just imagine three dimensional space over a time - that adds your fourth dimension.\\
Mathematician: No, it's way easier than that; just imagine \R{n}, then set $n$ = 4.}
\end{quote}

The humourous anectode strikes at the heart of the notion of `dimension'. With our Euclidean intuitions,  inherited from ancient primates, it is easy to jump to the notion that the physical universe we exist in can be well represented by a 3 parameter model of depth, breadth and height - that it is 3-dimensional. But what does such a statement really mean?

Modern mathematics has a way to formalize, generalize and give meaning to these notions by defining the Euclidean space as an $n$-dimensional real vector space. The vectors correspond to the points, and the addition operation in the vector space corresponds to translations in the Euclidean space. Defined in this manner, the plane is a 2-dimensional real vector pace \R{2}, the space is a 3-dimensional real vector space \R{3}, and so on. 

Given this context, the idea of having $R^n$ (the set of all $n$-tuples of real numbers), with a different dimension (say $k$), may seem absurd at first glance. The gap in the intuition stems from thinking of dimension as purely residing in the properties of the set $R^n$, and not in the vector operations that define the relationship between members of the set. The space \R{3} has dimension 3 not because a point in \R{3} is described by 3 different coordinates, but because it has imbibed certain special properties because of the way vector addition and scalar multiplication have been defined over it.\footnote{The idea of independence is crucial here. Space has three dimensions because the length of a box is independent of its width or breadth, and space-time is four-dimensional because the location of a point in time is independent of its location in space. Thus, space is three-dimensional because every point in space can be described by a linear combination of three independent vectors.}

We show that it is possible to change the structure of the space \R{n} to one having an arbitrary dimension $k\in\mathbb{N}$ (for all $n\in\mathbb{N}$). We do this by `naturally relating' it to the space \R{k}, in particular  redefining vector addition and scalar multiplication on $R^n$ to imitate those of the space \R{k}. We then generalize the result to any finite-dimensional vector space which exists in bijection with another vector space of a different dimension.

\section{On \R{n} being numerically equivalent to \R{k}}
It is a well known theorem in Linear Algebra that every $n$-dimensional vector space over a field $\mathbb{F}$ is isomorphic to the standard space $\mathbb{F}^n$. Thus, if $R^n$ were to attain the dimension $k$ over $\mathbb{R}$, the resulting space would be isomorphic to \R{k}. To this end, we first prove that the space \R{n} is numerically equivalent to the space \R{k}, which is a necessary condition for such an isomorphism to exist.\footnote{The counterintuitive result that cardinality did not respect dimensions was first discovered by Cantor, and led to his now famous remark, ``I see it, but I don’t believe it." For a better exposition, read \cite{gouvea2011cantor}}.

We first prove that the the unit square $\left[0,1\right]\times\left[0,1\right]$ is numerically equivalent to the unit interval $\left[0,1\right]$. This would suffice to prove that $\mathbb{R}$ is numerically equivalent to \R{2}, since a bijection from $\left[0,1\right]\mapsto\mathbb{R}$ exists. Also, it does not really matter whether we consider $\left[0,1\right]$, $\left(0,1\right]$, or $\left(0,1\right)$, since there are bijections between all of these.

\begin{theorem}
The set  $\left(0,1\right]\times\left(0,1\right]$ is numerically equivalent to the unit interval $\left(0,1\right]$
\end{theorem}

\begin{proof}
For real numbers with two decimal expansions, we will choose the one that ends with nines rather than with zeroes. Thus, we represent 1/2 as $0.4999...\left(0.4\overline{9}\right)$ instead of as 0.5.
Let $x = \left(a,b\right) \in$ \R{2}. Break each coordinate into groups consiting of zeroes (possibly none) followed by a single non-zero digit. For example, 1/400 = 0.0024999... is broken up as 0.002 4 9 9 9 ..., and 0.003801007373... as 0.003 8 001 007 3 7 3 ... This is well defined since we have not taken up decimal representations ending with infinite zeroes.
Thus now $x = (0.a_1a_2a_3\ldots, 0.b_1b_2b_3...) \in$ \R{2}, where $a_i$ and $b_j$ represent groups consisting of zeroes followed by a non-zero digit. Let $y = 0.a_1b_1a_2b_2...$ Note that $y$ will not have an infinite sequence of trailing zeroes. The resuling map $\phi : \left(0,1\right]\times\left(0,1\right] \mapsto \left(0,1\right]$, where $\phi\left(x\right)=y$, defines a bijection between the two sets. \footnote{We can also interleave digits, instead of groups as shown. However, this map would be not be surjective, and we'll need to use the Schroeder-Bernstein theorem to show the existence of the bijection. Such a proof would, however, be non-constructive. Cantor himself had originally tried a proof with interleaving digits, but later Dedekind  pointed out the problem of nonunique decimal representations.}
\end{proof}

\begin{corollary}\label{rr}
$\mathbb{R}$ is numerically equivalent to \R{2}. 
\end{corollary}

\begin{theorem}
The sets $R^n$ and $R^n+1$ are numerically equivalent, where $n\in\mathbb{N}$.
\end{theorem}

\begin{proof}
We’ll proceed by the principal of Mathematical Induction.\\
\textit{Base Case:} For $n=1$: $R$ and $R^2$ are numerically equivalent. (Corollary \ref{rr})\\
\textit{Inductive Hypothesis}: The sets $R^{n-1}$ and $R^n$ are numerically equivalent, where $n\in\mathbb{N}$.\\
\textit{Claim: }The sets $R^{n}$ and $R^{n+1}$ are numerically equivalent, where $n\in\mathbb{N}$.\\
	Let $\Phi:R^{n-1} \mapsto R^n$ be a bijection (Inductive hypothesis)\\
	If $x \in R^{n-1};$ $x = (x_1,...,x_{n-1})$, then ${\Phi(x)} = (\Phi(x)_1,...,\Phi(x)_n);$ ${\Phi(x)}\in R^n$.\\
	Take ${y} = (y_1,...,y_{n}) \in R^n$. The $(n-1)$-tuple formed from this $n$-tuple, by taking the first $n-1$ components is $y(n-1) = (y_1,...,y_{n-1});$ $y(n-1) \in R^{n-1}$.\\
	Let the element formed by appending $y_n$ to the $n$-tuple $ \Phi(y(n-1)) $ be denoted by $ \tau(y) $; $ \tau(y)  = (\Phi(y(n-1))_1,..., \Phi(y(n-1))_{n}, y_n )$.\\ The function $\tau:R^{n} \mapsto R^{n+1}$ where the element $y \in R^n$ is mapped to $\tau(y) \in R^{n+1}$, is a one-one onto correspondence.

\end{proof}
\begin{corollary}
The sets $R$ and $R^n$ are numerically equivalent, where $n\in\mathbb{N}$.
\end{corollary}
\begin{corollary}
The sets $R^n$ and $R^k$ are numerically equivalent, where $n,k\in\mathbb{N}$.
\end{corollary}

\section{The Vector Space $R^n_k$}
Let $\Phi:R^n \mapsto$ \R{k} be a bijection from the set of real n-tuples $R^n$ to the vector space \R{k}. As proved in the previous section, such a bijection exists, for all $n\in\mathbb{N}$ and all $k\in\mathbb{N}$.

Assume $\V{x}, \V{y}\in R^n$ and $c\in\mathbb{R}$. We define two binary operations, vector additon and scalar multiplication, on the set $R^n$ as follows:
\begin{align}
\text{\textbf{Vector Addition}} \colon \V{x}  \oplus \V{y} & = \Phi^{-1}(\Phi(\V{x}) + \Phi(\V{y}))\label{eq:1}\\
\text{\textbf{Scalar Multiplication}} \colon c\odot\V{x} &= \Phi^{-1}(c\cdot(\Phi(\V{x}))
\label{eq:2}
\end{align}

%From these definitions, we get:
%\begin{align}
%\Phi(\V{x}  \oplus \V{y}) & = \Phi(\Phi^{-1}(\Phi(\V{x}) + \Phi(\V{y})))  = \Phi(\V{x}) + \Phi(\V{y}) \label{eq:3}\\
%\Phi(c\odot\V{x}) &= \Phi(\Phi^{-1}(c\cdot(\Phi(\V{x}))) = c\cdot(\Phi(\V{x})
%\label{eq:4}
%\end{align}
\subsection{Axioms}

To qualify as a vector space, the set $R^{n}$ with the operations vector addition and scalar multiplication defined above must satisfy the following eight axioms.\\
%\allowdisplaybreaks
\begin{enumerate}

\item \textbf{Commutativity of addition}

\begin{align*}
& \,\quad \V{x}  \oplus \V{y} \\
& = \Phi^{-1}(\Phi(\V{x}) + \Phi(\V{y})) \tag{From Equation \ref{eq:1}}\\ 
& = \Phi^{-1}(\Phi(\V{y}) + \Phi(\V{x}))  \tag{Commutativity of \R{k}}\\  
%& = \Phi^{-1}(\Phi(\V{y}+\V{x})) \\
& = \V{y} \oplus \V{x} \tag{From Equation \ref{eq:1}}
\end{align*}
\item \textbf{Associativity of addition}
\begin{align*}
& \,\quad \V{x} \oplus (\V{y} \oplus \V{z}) \\
& = \V{x} \oplus (\Phi^{-1}(\Phi(\V{y}) + \Phi(\V{z}))) \tag{From Equation \ref{eq:1}}\\
& = \Phi^{-1}(\Phi(\V{x}) + \Phi((\Phi^{-1}(\Phi(\V{y}) + \Phi(\V{z}))))) \tag{From Equation \ref{eq:1}}\\
& = \Phi^{-1}(\Phi(\V{x}) + (\Phi(\V{y}) + \Phi(\V{z}))) \tag{$\Phi \circ \Phi^{-1} = \mathds{1}$}\\
& = \Phi^{-1}((\Phi(\V{x}) + \Phi(\V{y})) + \Phi(\V{z}))\tag{Associativity of \R{k}} \\
& = \Phi^{-1}((\Phi(\Phi^{-1}(\Phi\V{x} + \Phi\V{y})) + \Phi(\V{z})) \tag{$\Phi \circ \Phi^{-1} = \mathds{1}$}\\
& = \Phi^{-1}((\Phi(\V{x} + \V{y}) + \Phi(\V{z})) \tag{From Equation \ref{eq:1}} \\
& = (\V{x} + \V{y}) \oplus \V{z} \tag{From Equation \ref{eq:1}}
\end{align*}
\item \textbf{Existence of Identity element of addition}
\begin{align*}
 & \,\quad \V{x} \oplus \Phi^{-1}(0_k) \\
& =\Phi^{-1}(\Phi(\V{x}) + \Phi(\Phi^{-1}(0_k))) \tag{From Existence of Identity $0_k$ for \R{k}}\\
& =\Phi^{-1}(\Phi(\V{x}) + 0_k) \tag{$\Phi \circ \Phi^{-1} = \mathds{1}$}\\
& = \Phi^{-1}(\Phi(\V{x})) \tag{ $0_k$ is Identity for \R{k}}\\
& = \V{x} \tag{Hence Additive Identity for $R^{n}_{k}$ is $0_{k}$}
\end{align*}
\item \textbf{Existence of Additive Inverse}
\begin{align*}
& \,\quad \V{x} \oplus \Phi^{-1}(-\Phi(\V{x})) \\
& = \Phi^{-1}(\Phi(\V{x}) + \Phi(\Phi^{-1}(-\Phi(\V{x}))) \tag{From Equation \ref{eq:1}}\\
& = \Phi^{-1}(\Phi(\V{x}) + (-\Phi(\V{x}))) \tag{$\Phi \circ \Phi^{-1} = \mathds{1}$}\\
& = \Phi^{-1}(0_k) \tag{Hence Additive Inverse of $\V{x}$ is $\Phi^{-1}(-\Phi(\V{x}))$}
\end{align*}
\item \textbf{Existence of Identity element of scalar multiplication}
\begin{align*}
 & \,\quad 1\odot\V{x} \\
 & = \Phi^{-1}(1\cdot(\Phi(\V{x})) \tag{From Equation \ref{eq:2}}\\ 
 & = \Phi^{-1}(\Phi(\V{x}))  \tag{Multiplication by Scalar Identity for \R{k}}\\  
 & = \V{x} \tag{$\Phi \circ \Phi^{-1} = \mathds{1}$}\\ 
 \end{align*}
\item \textbf{Compatibility of scalar multiplication with field multiplication}
\begin{align*}
 & \,\quad  c_1\odot(c_2\odot\V{x})   \\
& = \Phi^{-1}(c_1\cdot(\Phi(c_2\odot\V{x})) \tag{From Equation \ref{eq:2}}\\
& = \Phi^{-1}(c_1\cdot(\Phi(\Phi^{-1}(c_2\cdot\V{x}))) \tag{From Equation \ref{eq:2}}\\
 & = \Phi^{-1}(c_1\cdot(c_2\cdot\Phi(\V{x})))  \tag{$\Phi \circ \Phi^{-1} = \mathds{1}$}\\  
 & = \Phi^{-1}(c_1c_2\cdot\Phi(\V{x})))  \tag{Compatibility of \R{k}}\\ 
& = c_1c_2\odot\V{x} \tag{From Equation \ref{eq:2}}\\
\end{align*}
\item \textbf{Distributivity of scalar multiplication with respect to vector addition}
\begin{alignat*}{2}
 & \,\quad c\odot(\V{x} \oplus \V{y}) \\
 & = \Phi^{-1}(c\cdot\Phi(\V{x} \oplus \V{y}) ) \tag{From Equation \ref{eq:2}}\\
 & = \Phi^{-1}(c\cdot\Phi(\Phi^{-1}(\Phi(\V{x}) + \Phi(\V{y}))) ) \tag{From Equation \ref{eq:1}}\\
 & = \Phi^{-1}(c\cdot(\Phi(\V{x}) + \Phi(\V{y})) ) \tag{$\Phi \circ \Phi^{-1} = \mathds{1}$}\\  
& =\Phi^{-1}(c \cdot \Phi(\V{x}) + c\cdot \Phi(\V{y}) ) \tag{From Distributivity of \R{k}}\\
& =\Phi^{-1}(\Phi(\Phi^{-1}(c \cdot \Phi(\V{x})) + \Phi(\Phi^{-1}(c\cdot \Phi(\V{y}) )) \tag{$\Phi \circ \Phi^{-1} = \mathds{1}$}\\  
& = \Phi^{-1}( \Phi(c\odot\V{x}) + \Phi(c\odot\V{y}) )  \tag{From Equation \ref{eq:2}}\\
& = (c\odot\V{x}) \oplus (c\odot\V{y})  \tag{From Equation \ref{eq:1}}\\
\end{alignat*}
\item \textbf{Distributivity of scalar multiplication with respect to field addition}
\begin{align*}
& \,\quad (c_1 + c_2) \odot \V{x}\\
& =  \Phi^{-1}((c_1 + c_2) \cdot (\Phi(\V{x})) \tag{From Equation \ref{eq:2}}\\
& =  \Phi^{-1}(c_1\cdot (\Phi(\V{x}) + c_2\cdot (\Phi(\V{x})) \tag{From Distributivity of \R{k}}\\
%& =  \Phi^{-1}(\Phi(\Phi^{-1}(c_1\cdot (\Phi(\V{x})) + \Phi(\Phi^{-1}(c_2\cdot (\Phi(\V{x})) )\tag{$\Phi \circ \Phi^{-1} = \mathds{1}$}\\  
%& =  \Phi^{-1}(\Phi(c_1\odot\V{x}) + \Phi(c_2\odot\V{x})) \tag{From Equation \ref{eq:2}}\\
%& =  (c_1\odot\V{x}) + (c_2\odot\V{x}) \tag{$\Phi \circ \Phi^{-1} = \mathds{1}$}
\end{align*} 
\end{enumerate}

The set $R^n$, together with vector additon and scalar multiplication as defined in Equation \ref{eq:1} and \ref{eq:2}, satisfies all eight axioms of a vector space over the field $\mathbb{R}$. We call this vector space $R^{n}_{k}$.
\subsection{Basis and Dimension} 

What is the dimension of the vector space $R^{n}_{k}$.?

Let $\mathcal{B} = \lbrace \alpha_1, \ldots, \alpha_k \rbrace $ be an ordered basis for \R{k}. Let $T \subset R^n \colon T = \lbrace \Phi^{-1}(\alpha_1), \ldots, \Phi^{-1}(\alpha_k) \rbrace$.

Take $\V{x} \in R^{n}$. Since $\mathcal{B}$ is a basis for \R{k}, and $\Phi(\V{x}) \in$ \R{k}, there exists a unique set of real numbers (coordinates) $(x_1, \ldots, x_k)$ such that:
\begin{align*}
\Phi(\V{x}) =  x_1\cdot\alpha_1 + \ldots + x_k\cdot\alpha_k \\
\Leftrightarrow \V{x} = \Phi^{-1}(x_1\cdot\alpha_1 + \ldots + x_k\cdot\alpha_k) \\
\Leftrightarrow \V{x} = \Phi^{-1}(x_1\cdot\Phi(\Phi^{-1}\alpha_1) + \ldots + x_k\cdot\Phi(\Phi^{-1}\alpha_k)) \\
\Leftrightarrow \V{x} = \Phi^{-1}( \Phi(x_1\odot\Phi^{-1}\alpha_1) + \ldots +\Phi( x_k\odot\Phi^{-1}\alpha_k ) ) \\
%\Leftrightarrow \V{x} = \Phi^{-1}(\Phi(x_1\odot\Phi^{-1}\alpha_1 + \ldots + x_k\odot\Phi^{-1}\alpha_k) ) \\
\Leftrightarrow \V{x} = (x_1\odot\Phi^{-1} \alpha_1) \oplus \ldots \oplus (x_k\odot\Phi^{-1}\alpha_k) \\
\end{align*}
Thus, for every $\V{x} \in R^{n}$, there exists a unique set of real numbers  $(x_1, \ldots, x_k)$ such that $ \V{x} = x_1\odot\Phi^{-1} \alpha_1 + \ldots + x_k\odot\Phi^{-1}\alpha_k$. This implies that $ \Phi^{-1}(\mathcal{B})\subset R^n = \lbrace \Phi^{-1}(\alpha_1), \ldots, \Phi^{-1}(\alpha_k) \rbrace $ is a basis for $R^{n}_{k}$. Since $|\Phi^{-1}(\mathcal{B})| = k$, $dim\, R^{n}_{k} = k$.\footnote{The dimension of a vector space is the same as the cardinality of its basis}

\section{Conclusion}

Let $\Phi:V \mapsto W$ be a one-one onto correspondence, where $W$ is a vector space over field $\mathbb{F}$. Then we can always define vector addition and scalar multiplication on the set $V$, in such a way, so that it attains the vector space structure and dimension of $W$. This would turn $\Phi$ into a linear map between the two vector spaces, $V$ and $W$.
In particular, if $W$ is a $k$-dimensional Euclidean space, then we can always give the set \R{n} the structure of the $k$-dimensional space. The set $ \left\{(x_1,x_2,x_3): x_i \in \mathbb{R}, i= 1,2,3\right\}$ for instance, can always be represented as having dimension of 1 over $\mathbb{R}$.

\section*{Appendix}

\begin{theorem} \label{{easy}}
The sets $\left[0,1\right], \left(0,1\right]$ and $(0,1)$ are in bijection with each other.
\end{theorem}

\begin{proof}
Let $f$ be the function with mappings $0\mapsto \frac12, \frac12\mapsto\frac23,\frac23\mapsto\frac34,$ and so on. For any other $x \in [0,1] -\left\{0, \frac12, \frac23, \frac34,\ldots\right\}$, $f\left(x\right) = x$. Then $f:[0,1]\mapsto(0,1]$ is a bijection. A bijection from $\left(0,1\right] to \left(0,1\right)$ can be found similarly.
\end{proof}

\begin{theorem}\label{figthm}
There is a bijection from $\mathbb{R}$ to the open interval $(0,1)$.
\end{theorem}
\begin{proof} 
\includegraphics[scale=0.5,trim=0 0 0 0,clip=]{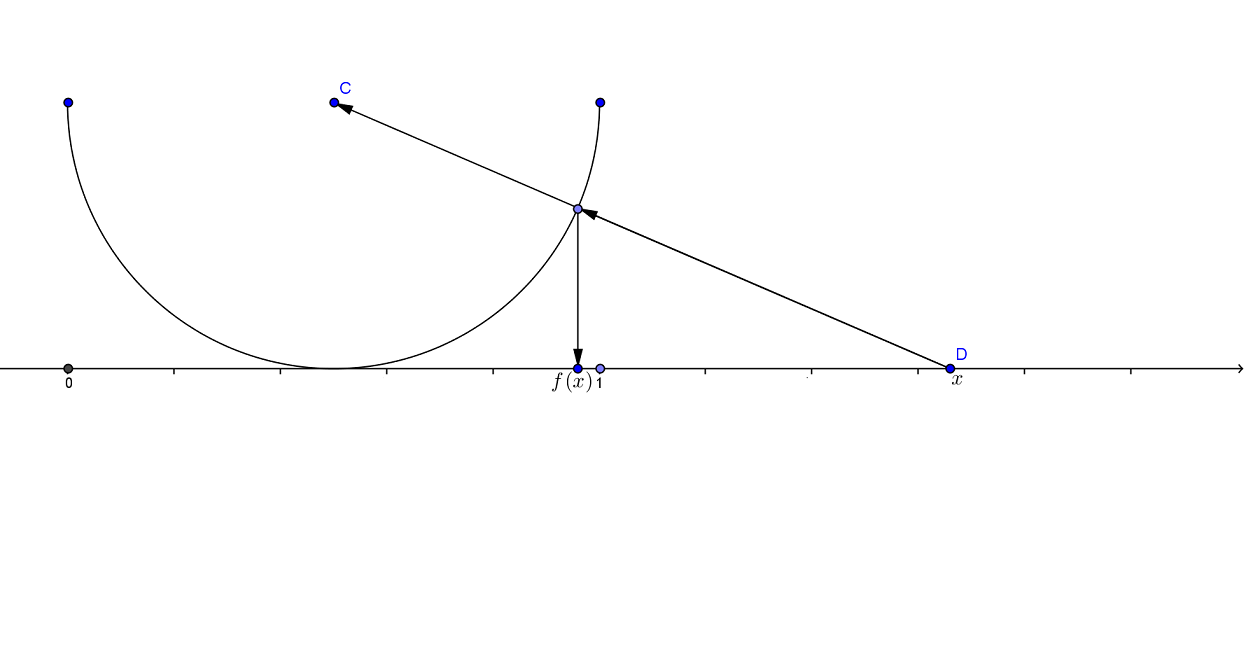}\label{fig:main}
In Figure \ref{fig:main}, $S$ is a semi-circle, excluding the end points. $S$ is drawn in such a way that the real axis is tangent to it at 0.5, with radius 0.5. An arbitrarily chosen point $x \in \mathbb{R}$ is joined to the centre of $S$ with a line. The point $x'$ is obtained on $S$ by the intersection of the line and the semi-circle. This gives a one-to-one onto correspondence between the every point $x \in \mathbb{R}$ and $x' \in S$. Now project the point $x' \in S$ to the point $f\left(x\right)$ on the real axis. The function $f$ thus defined gives a bijection between $\mathbb{R}$ and open interval $(0,1)$.

\end{proof}

\begin{theorem}
$R^n_k$ is isomorphic to the space \R{k}.
\end{theorem}

\begin{proof}
Note that:
\begin{eqnarray}
\Phi(\V{x}\oplus\V{y}) &= \Phi(\Phi^{-1}(\Phi(\V{x}) + \Phi(\V{y}))) &= \Phi(\V{x}) + \Phi(\V{y})
\label{eq:3}\\
\Phi(c\odot\V{x}) &= \Phi(\Phi^{-1}(c\cdot(\Phi(\V{x}))) &= c\cdot(\Phi(\V{x}))
\label{eq:4}
\end{eqnarray}
Hence, the function $\Phi$ is a linear transformation as well as a bijection from $R^n_k$ to \R{k}.
Thus, $R^n_k \cong$ \R{k} 
\end{proof}

\begin{corollary}
The space $R^n_k$, defined as above, is a vector space over the field $\mathbb R$ with dimension $k$.
\end{corollary}

\section*{Acknowledgement}
We are grateful to Ashutosh Nayyar, for the coming up with the initial idea which eventually sparked off the result. 

\bibliographystyle{amsmath}
\nocite{*}

\bibliography{bibliography.bib}
\end{document}